\documentclass{mcom-l2}

\theoremstyle{plain}
\newtheorem{thm}{Theorem}[section]
\newtheorem{prop}[thm]{Proposition}
\newtheorem{cor}[thm]{Corollary}
\newtheorem{heur}[thm]{Heuristic}
\newtheorem*{heur*}{Heuristic}
\newtheorem{lem}[thm]{Lemma}
\theoremstyle{definition}
\newtheorem{defin}[thm]{Definition}
\theoremstyle{remark}
\newtheorem{ex}[thm]{Example}
\newtheorem{rem}[thm]{Remark}

\usepackage{algorithm,algorithmic}

\algsetup{indent=2em}

\usepackage{tikz,multirow,url}
\usetikzlibrary{decorations.pathreplacing}

\newcommand{\CC}{{\mathbf C}}

\newcommand{\KK}{{\mathbf K}}
\newcommand{\NN}{{\mathbf N}}
\newcommand{\OO}{{\boldsymbol{\mathcal O}}}
\newcommand{\QQ}{{\mathbf Q}}
\newcommand{\RR}{{\mathbf R}}
\newcommand{\ZZ}{{\mathbf Z}}
\newcommand{\Bcal}{{\mathcal B}}

\newcommand{\Dcal}{{\mathcal D}}
\newcommand{\Ical}{{\mathcal I}}
\newcommand{\Lcal}{{\mathcal L}}
\newcommand{\Ncal}{{\mathcal N}}
\newcommand{\Pcal}{{\mathcal P}}
\newcommand{\Ucal}{{\mathcal U}}
\newcommand{\afrak}{{\boldsymbol{\mathfrak a}}}
\newcommand{\bfrak}{{\boldsymbol{\mathfrak b}}}
\newcommand{\cfrak}{{\boldsymbol{\mathfrak c}}}
\newcommand{\pfrak}{{\boldsymbol{\mathfrak p}}}

\newcommand{\Cl}{{\textit{\textrm{Cl}}\left(\OO_\KK\right)}}

\newcommand{\Reg}{{\textit{\textrm{Reg}}}}
\newcommand{\DD}{{|\Delta_\KK|}}

\newcommand{\ie}{{\emph{i.e.,}~}}
\newcommand{\ve}{{\varepsilon}}
\newcommand{\DDi}{{|\Delta_{\KK_i}|}}

\newcommand{\Poly}{{\operatorname{Poly}}}

\newcommand{\gen}[1]{\left\langle {#1} \right\rangle}

\newcommand{\LL}{\textit{\L}}

\title[Class group computations for large degree number fields]{On the complexity of class group computations for large degree number fields}

\author{Alexandre G\'elin}
\address{Laboratoire de Math\'ematiques de Versailles, UVSQ, CNRS, Universit\'e Paris-Saclay, Versailles, France}
\email{alexandre.gelin@uvsq.fr}

\begin{document}

\begin{abstract}
In this paper, we examine the general algorithm for class group computations, when we do not have a small defining polynomial for the number field. Based on a result of Biasse and Fieker, we simplify their algorithm, improve the complexity analysis and identify the optimal parameters to reduce the runtime. We make use of the classes~$\mathcal D$ defined in~\cite{GJ16} for classifying the fields according to the size of the extension degree and prove that they enable to describe all the number fields. 
\end{abstract}

\maketitle


\section{Introduction}

In algebraic number fields, two structures are of particular interest: the class group, finite, and the unit group, finitely generated. Their computations are main problems in algorithmic algebraic number theory. Shanks~\cite{Sha69,Sha72} first described an algorithm, the baby-step--giant-step method, in the special case of quadratic number fields. This method runs in exponential runtime $O(\DD^{\frac 1 5})$ under the General Riemann Hypothesis (GRH), where $\Delta_\KK$ denotes the absolute discriminant of the considered number field.

For imaginary quadratic number fields, Hafner and McCurley~\cite{HMC89} managed to compute the class group structure in heuristic subexponential time $L_\DD(\frac 1 2, \sqrt{2})$. 
This $L$-notation is classical when presenting index calculus algorithms with subexponential complexity. Given two constants $\alpha$ and $c$ with $\alpha\in [0,1]$ and $c \geq 0$, $L_N(\alpha,c)$ is used as a shorthand for:
\[\exp\left((c + o(1))(\log N)^\alpha (\log \log N)^{1-\alpha} \right),\]
where $o(1)$ tends to $0$ as $N$ tends to infinity. We sometimes encounter the notation $L_N(\alpha)$ when specifying $c$ is superfluous.

Buchmann~\cite{Buc90} extended this method to all number fields. However, the extension degree, arbitrary, has to be fixed to obtain the heuristic complexity $L_\DD(\frac 1 2, 1.7)$. More recently, the subexponential complexity was reached for all number fields, without restriction on the extension degree. Biasse and Fieker~\cite{BF14} got an $L_\DD(\frac{2}{3}+\varepsilon)$ complexity\footnote{For an arbitrary small $\varepsilon>0$.} in the general case and $L_\DD(\frac{1}{2})$ when the extension degree $n$ satisfies the inequality $n\leq (\log\DD)^{3/4-\varepsilon}$. 

For some restricted classes of number fields, Biasse and Fieker~\cite{BF14} achieved an even better~$L_\DD(a)$ complexity with $a$ possibly as low as $\frac{1}{3}$. More precisely, this improved complexity holds when one knows a defining polynomial with small coefficients compared to the discriminant of the field. New classes of number fields have been introduced in~\cite{GJ16} in order to widen the conditional improvement of Biasse and Fieker by looking for such a small defining polynomial. 

\subsection*{Contribution.} We first show that regarding at the classes~$\Dcal$ introduced in~\cite{GJ16} suffices to consider all number fields. This enables to give a bird's eye view of the state of the art concerning class group computations, according to the extension degree of the number field. We then focus on large degree number fields: we give a simplified version of the relation collection and, thanks to a better choice for the parameters, show that it can run in time $L_\DD \left(a \right)$ with $a \in [\frac{1}{2}, \frac{2}{3}]$ instead of $L_\DD \left(\frac{2}{3} + \varepsilon \right)$ when the extension degree is large. In addition, we refine the $L_\DD \left(\frac{1}{2} \right)$ complexity by calculating the second constant: we obtain a runtime in $L_\DD \left(\frac{1}{2}, \frac{\omega-1}{2\sqrt{\omega}} \right)$. At the very end, using another enhancement on lattice reduction, we present an improved version whose complexity grows linearly between $L_\DD \left(\frac{1}{2}\right)$ and $L_\DD \left(\frac{3}{5}\right)$ instead of $L_\DD \left(\frac{2}{3}\right)$.

\subsection*{Outline.} The article is organized as follows. In Section~\ref{sec:genst} we provide a reminder about index calculus method, applied in the context of class group computation. Then we classify number fields according to the classes~$\Dcal$ in Section~\ref{sec:class}. Finally, the algorithm is described in Section~\ref{sec:algo} while Section~\ref{sec:Comp} is devoted to the complexity analysis. The last improvement is the topic of Section~\ref{sec:cheon}.

\section{General strategy for class group computation} \label{sec:genst}

The current best algorithms for class group computation rely on the index calculus method. It is also the case for factoring integers or computing discrete logarithms in finite fields. A brief summary is as follows:

\begin{enumerate}
\item Fix a factor base composed of small elements and that is large enough to generate all elements of the group.

\item Collect relations between those small elements, corresponding to linear equations.

\item Deduce the result sought performing linear algebra on the system built from the relations.
\end{enumerate}

We give more details about the different steps in case of class group computation. Afterwards, every contribution is examined with respect to this global strategy.

\subsection*{The factor base.}
We define the factor base $\Bcal$ as the set of all prime ideals in~$\OO_\KK$ that have a norm bounded by a constant $B$. This bound must be chosen such that the factor base generates the whole class group. Bach showed in~\cite[Theorem~4.4]{Bac90} that assuming the Extended Riemann Hypothesis (ERH), the classes of ideals with a representative of norm less than $12 \left(\log \DD\right)^2$ suffice to generate the class group.
However, as the ability to find relations in the collection step increases with the size of the factor base, we fix \[B= \LL_\DD(\beta, c_b),\]
for values of $\beta$ and $c_b$ with $0 < \beta < 1$ and $c_b >0$ that are determined later. The notation $\LL$ is identical as the $L$ introduced earlier, except that we have removed the $o(1)$, in order to consider constants: $\LL_N(\alpha,c) = e^{c\left(\log N\right)^\alpha\left(\log \log N\right)^{1-\alpha}}$.

Thanks to the Landau Prime Ideal Theorem~\cite{Lan03}, we know that in every number field $\KK$, the number of prime ideals of norm bounded by $B$, denoted by~$\pi_\KK(B)$, satisfies
\begin{equation} \label{eq:Landau}
\pi_\KK(B) \sim \frac{B}{\log B}.
\end{equation}
As a consequence, the cardinality of the factor base is about $B$, namely: \[N=|\Bcal| = L_\DD(\beta, c_b).\]

\subsection*{Relation collection.}
Let $\pfrak_i$, $1 \leq i \leq N,$ denote the $N$ prime ideals in the factor base $\Bcal$. As their classes generate the class group $\Cl$, we have a surjective morphism \mbox{$\phi:\ZZ^n \longrightarrow \Cl$} via
\begin{equation} \label{eq:Morph}
\begin{tabular}{ccccc}
$\ZZ^N$ & $\longrightarrow$ & $\Ical$ & $\longrightarrow$ & $\Cl$ \\
$(e_1,\dotsc,e_N)$ & $\longmapsto$ & $\prod\limits_i \pfrak_i ^{e_i}$ & $\longmapsto$ & $\prod\limits_i [\pfrak_i] ^{e_i}$,
\end{tabular}
\end{equation}
and the class group $\Cl$ is then isomorphic to $\ZZ^N / \ker \phi$. By computing the kernel of this morphism, we deduce the class group, which is given by the lattice of the vectors $(e_1,\dotsc,e_N)$ in~$\ZZ^N$ for which $\prod \pfrak_i ^{e_i} = \gen{x}$ with $x \in \KK^*$. Thus the relations that we want to collect are given by~$x$ in $\KK^*$ such that 
\begin{equation} \label{eq:GenVal}
\gen{x} = \prod \pfrak_i ^{e_i}. 
\end{equation}

Relation collection is the main part of the algorithm, we give more details about~it in Section~\ref{sec:algo}.

\subsection*{Linear algebra.}
Once the relations are collected, we store them in a matrix. A row corresponds to an algebraic number $x$ and the $i$-th coefficient is the valuation of the principal ideal~$\gen{x}$ at $\pfrak_i$ --- that is $e_i$ in Equation~\eqref{eq:GenVal}. These valuations $e_i$ are computed by looking first at the norm of $\gen{x}$, as explained in Appendix~\ref{app:smooth}. Then, the structure of the class group is given by the \emph{Smith Normal Form} (SNF) of the matrix. More precisely, we first compute the \emph{Hermite Normal Form} (HNF) with a pre-multiplier since we need kernel vectors in the verification step (as explained below). Finally, the class number can be deduced by multiplying the diagonal coefficients of the HNF while the group structure is given by the diagonal coefficients of the SNF.

\subsection*{Verification.}
The group $H$ provided by the linear algebra step is only a candidate for the class group and has to be verified. Indeed, even assuming that the factor base is large enough to generate the full class group, the number of relations derived may be insufficient. In that case, the class group $\Cl$ is only a quotient of the candidate $H$. Fortunately we can obtain some information on the class number from the Class Number Formula:

\begin{prop}[{\cite[Theorem~4.9.12]{Coh93}}]
Let $\KK$ be a number field of degree~$n$ with $n=r_1+2r_2$ where $r_1$ denotes the number of real embeddings and $r_2$ the number of pairs of complex embeddings. Let $h_\KK$, $\Reg_\KK$, $\Delta_\KK$, $w_\KK$ and $\zeta_\KK(s)$ denote respectively the class number, the regulator, the discriminant, the number of roots of unity and the Dedekind zeta function of $\KK$. Then the function $\zeta_\KK(s)$ converges absolutely for~$s$ with $\Re(s) > 1$ and extends to a meromorphic function defined for all complex~$s$ with only one simple pole at $s = 1$, whose residue satisfies
\[ \lim_{s \to 1} (s-1)\zeta_\KK(s) = \frac{2^{r_1} \cdot (2\pi)^{r_2} \cdot h_\KK \cdot \Reg_\KK}{w_\KK \cdot \sqrt{\DD}}. \]
\end{prop}

We recall that this residue can also be expressed as the Euler Product:
\begin{equation} \label{PE}
\prod _p \frac {1-\frac 1 p } {\prod\limits_{\pfrak \mid p} \big(1-\frac 1 {\mathcal N(\pfrak)} \big)}~,
\end{equation}
product being taken over all prime numbers $p$. An approximation of this product may be computed at the very beginning of the algorithm, along with the generation of the factor base. Indeed Bach proves in~\cite{Bac95} that a good enough approximation is obtained in polynomial time considering only the primes of norm below $O\left((\log \DD)^2\right)$.

Thus we need at least an approximation of the regulator of the number field in order to carry out this verification. Fortunately, it does not cost too much to determine a candidate for the regulator once we have our candidate for the class group. Indeed, the collected relations make it possible to infer one: by looking for elements in the kernel of the relation matrix, we are computing units of $\KK$. Then once we have found generators of the group spanned by these units, it only remains to compute a determinant. If these generators form a set of fundamental units, we get the regulator. Otherwise, we have only found a multiple of the regulator, because the group spanned by those is a subgroup of the unit group $\Ucal(\KK)$.

In the end, when we have the --- hypothetical --- class number and regulator, it is enough to compare their product with the approximation calculated from the Euler Product. Either the ratio is close to 1 in which case the two quantities are the correct ones, or it is not and more relations are required. This verification step works since both class number and regulator are computed decreasingly: if there is a sufficient number of primes ideals --- respectively units --- involved, then adding a relation can only reduce the class number --- respectively the regulator~--- by an integer factor. As a consequence, the ratio is close to 1 only for $h_\KK$ and $\Reg_\KK$.

\section{The classification defined by classes $\Dcal$ is sufficient} \label{sec:class}

For the discrete logarithm problem in finite fields, all the fields are classified according to the relative size of their characteristic --- small, medium or large. Our purpose is to derive a similar classification for the number fields. For finite fields, the cardinality $Q$ is completely determined by the characteristic $p$ and the extension degree $n$, according to the equation $Q=p^n$. For number fields, the extension degree remains, but the characteristic is replaced by the size of the defining polynomial, represented by its height $H(T)$. Unfortunately, number fields do not provide any equality similar to $Q=p^n$ for finite fields, but only the inequality of~\cite[Proposition~2.1]{GJ16}:
\begin{equation} \label{eq:DiscH} \DD \leq n^{2n} H(T)^{2n-2}.\end{equation}

Therefore, we choose the extension degree as the main parameter of our classification. The Minkowski's bound~\cite[Corollary~2.9]{PZ89} induces that $n = O(\log \DD)$, because every non-zero integral ideal has a norm in $\NN^*$. Thus we want to express~$n$ in terms of $\log \DD$. Fortunately, this choice is a perfect match with the classes $\Dcal$ introduced in~\cite{GJ16}.

\begin{defin}[{\cite[Section~3]{GJ16}}]
Let $n_0>1$ be a real parameter arbitrarily close to $1$, $d_0 > 0$, \mbox{$\alpha\in [0,1]$} and $\gamma \geq 1 - \alpha$. The class $\Dcal_{n_0,d_0,\alpha,\gamma}$ is defined as the set of all number fields $\KK$ of discriminant $\Delta_\KK$ that admit a monic defining polynomial $T \in \ZZ[X]$ of degree $n$ that satisfies:
\begin{align} \label{eq:ClassD}
\frac{1}{n_0} \left(\frac{\log \DD }{\log\log \DD}\right)^\alpha \quad \leq \quad &n \quad \leq \quad n_0 \left(\frac{\log \DD }{\log\log \DD}\right)^\alpha \qquad \text{and} \nonumber\\
d = \log H(T) \quad &\leq \quad d_0 (\log \DD)^\gamma (\log\log \DD)^{1-\gamma}.
\end{align}
\end{defin}

We recall that the factor $\log \log \DD$ has been introduced to simplify the complexity analysis, while the condition $\gamma \geq 1 - \alpha$ is a direct consequence of Equation~\eqref{eq:DiscH}.
We emphasize that the extension degree carries more information than the size of the coefficients of a defining polynomial --- while giving the extension degree or the characteristic of a finite field carries the same information. Indeed, there exists an infinity of defining polynomials, and the quality of the smallest one depends on the number field: it is not known that we can always find one satisfying the lower bound $\gamma = 1 - \alpha$. That is why classifying number fields by their extension degree $n$ --- that is by $\alpha \in [0,1]$ --- makes more sense. Then, for each $\alpha$, there exists additional disparities according to $\gamma$, which is always greater than $1- \alpha$.

Here is the main theorem obtained in~\cite{GJ16}:
\begin{thm} \label{thm:ANTS}
Under ERH and smoothness heuristics, for every number field $\KK$ that belongs to $\Dcal_{n_0,d_0,\alpha,\gamma}$, there exists an $L_\DD(a,c)$ algorithm for class group and unit group computation for some $c > 0$ and $a=\max \left(\alpha,\frac \gamma 2 \right)$.
\end{thm}

Thanks to the algorithm described in~\cite{GJ16}, we can restrict our study to the classes~$\Dcal$ with $\gamma \leq 1$ when~$\alpha$ is in~$\left[0,\frac{1}{2}\right]$. In these cases, Theorem~\ref{thm:ANTS} shows that we can compute the class group in time below $L_\DD\left(\frac{1}{2}\right)$. When $\alpha \geq \frac{1}{2}$, it is too costly to look for a small polynomial. We focus in this article on \emph{large degree} number fields, the ones where $\alpha \geq \frac{1}{2}$. 

At this point, it still remains to prove that considering classes $\Dcal$ with $\alpha \in [0,1]$ suffices. At first sight, the Minkowski theorem only results in $n = O(\log \DD)$ and implies that every number field belongs to a class $\Dcal$ with $\alpha \leq 1 + \ve$ for an arbitrarily small $\ve > 0$. However, a more accurate analysis leads to the following result:

\begin{prop} \label{prop:NonEx}
Given $n_0 > 1$ and $\alpha > 1$, there does not exist an infinite family~$(\KK_i)_{i \geq 1}$ of number fields with discriminants $\DDi$ and degrees $n_i$ that satisfy
\[\frac{1}{n_0} \left(\frac{\log \DDi }{\log\log \DDi}\right)^\alpha \quad \leq \quad n_i \quad \leq \quad n_0 \left(\frac{\log \DDi }{\log\log \DDi}\right)^\alpha.\]
\end{prop}

\begin{proof}
We proceed by contradiction. Let $(\KK_i)_{i \geq 1}$ be an infinite family of number fields whose degrees $n_i$ satisfy
\[\frac{1}{n_0} \left(\frac{\log \DDi }{\log\log \DDi}\right)^\alpha \leq n_i.\]
We provide an upper bound in the statement of the proposition as it is in the definition of classes $\Dcal$. However, we only consider this inequality because it is the one that is problematic. The Minkowski's bound~\cite[Corollary~2.9]{PZ89} states that for a field $\KK$ of degree $n$,
\begin{equation} \label{eq:Mink}
\frac{n^n}{n!} \cdot \left(\frac{\pi}{4}\right)^{\frac{n}{2}} \leq \sqrt{\DD}.
\end{equation}

Combining Equation~\eqref{eq:Mink} with the inequality $n! \leq e\,n^{n+\frac12}\,e^{-n}$ derived from the Stirling formula~\cite{Moi30,Sti30}, we obtain 
$n \left(2+\log \frac{\pi}{4} \right) \leq \log \DD + 2 + \log n$. Let~$A$ denote the constant $2+\log \frac{\pi}{4} > 1$. Then for all $i \geq 1$, we have
\begin{multline*}
\frac{A}{n_0} \left(\frac{\log \DDi}{\log\log \DDi}\right)^\alpha \leq \log \DDi + 2 + \log n_0 + \alpha \left(\log\log \DDi-\log\log\log \DDi\right) \\
\Longrightarrow \quad 0 < \frac{A}{n_0} \leq \frac{\left(\log\log \DDi\right)^\alpha}{\left(\log \DDi\right)^{\alpha-1}} + \left(2 + \log n_0 + \alpha \log\log \DDi\right) \cdot \left(\frac{\log\log \DDi}{\log \DDi}\right)^\alpha.
\end{multline*}

Finally, as the set of number fields having bounded discriminant is finite, it follows from our initial assumption that the family of discriminants $\left(\DDi\right)_{i \geq 1}$ tends to infinity. But in that case, as $\alpha$ is chosen strictly greater than 1, the right hand side tends to 0, which leads to a contradiction.
\end{proof}

\begin{ex}
To illustrate this proposition, we consider cyclotomic fields, which are known to be fields with small discriminants and large degrees.

For the $l$-th cyclotomic field $\KK = \QQ(\zeta_l)$, with $l = \prod p_i^{k_i}$ and denoting by $\varphi$ the Euler totient function, the extension degree satisfies
\[\left[\QQ(\zeta_l) : \QQ\right] = \varphi(l) = \prod \varphi\left(p_i^{k_i}\right) = \prod (p_i-1)p_i^{k_i-1},\]
and the discriminant is (see~\cite[Proposition~2.7]{Was97})
\[\DD = \frac{l^{^{\varphi(l)}}}{\prod p_i^{{\varphi(l)}/{p_i-1}}}.\]

Thus we obtain
\begin{equation} \label{eq:phi}
\varphi(l) = \frac{\log\DD}{\log \log \DD} \cdot \frac{\sum (k_i-1) \log p_i + \log (p_i-1)}{\sum (k_i-\frac{1}{p_i-1})\log p_i} \big(1+o(1)\big),  
\end{equation}
and as $(k_i-1) \log p_i + \log (p_i-1) \approx (k_i-\frac{1}{p_i-1})\log p_i$ when $p_i$ or $k_i$ tends to infinity, we conclude that the ratio of the sums tends to 1 when $l$ tends to infinity.

For instance, when $l=p$, the second factor in Equation~\eqref{eq:phi} is $\frac{p-1}{p-2} \frac{\log (p-2)+ \log \log p}{\log p}$, which tends to 1 as $p$ goes to infinity, while for $l=p^k$ with $p$ fixed and $k$ tending to infinity, the second factor becomes $\frac{k}{k-\frac{1}{p-1}} \left(1+o(1)\right)$. 
\end{ex}

Hence all cyclotomic fields asymptotically belong to a class $\Dcal$ with $\alpha = 1$. Finally, Proposition~\ref{prop:NonEx} leads to the following statement: 

\begin{cor} \label{cor:ClassD}
Asymptotically, the classes $\Dcal_{n_0,d_0,\alpha,\gamma}$ with $\alpha \in [0,1]$ include all number fields.
\end{cor}

Note that despite~\cite[Corollary~3.3]{GJ16}, we do not specify the condition $\gamma$ in $[1-\alpha,1]$ in this result. Indeed when $\alpha \geq \frac{1}{2}$, finding the smallest height defining polynomial costs more than computing the class group. In these cases, it is preferable to work with the input polynomial. Another possibility is to perform only a partial reduction. More precisely, we may use the reduction algorithm described in~\cite[Section 4.4]{Coh93} which consists in computing an LLL-reduced basis of the lattice of algebraic integers. Assuming that an integral basis is already known, the runtime is polynomial in $\log \DD$. Eventually, for the reminder of the article, we focus our study on classes~$\Dcal$ with $\alpha \in \left[\frac{1}{2},1\right]$ and $\gamma \geq 1-\alpha$. Indeed, although the algorithm works for~$\alpha \leq \frac{1}{2}$, the complexity is larger than what is stated in~\cite{GJ16}.

\section{The relation collection} \label{sec:algo}

The core idea is presented by Biasse in~\cite{Bia14b}: the generation of the relations based on BKZ-reductions of ideal lattices. The strategy is still the same as Buchmann's work~\cite{Buc90}: we reduce an ideal $\afrak$, using lattice techniques, in order to find another ideal $\bfrak$ that belongs to the same class. While the algorithm of Buchmann looks for a shortest non-zero vector --- whose runtime is polynomial in the size of the discriminant but exponential in the extension degree --- the method of Biasse involves BKZ-reductions, that offer a trade-off between the time spent in the reduction and the approximation factor of the short vectors. This leads to a subexponential algorithm that allows both the discriminant and the degree to tend to infinity. When combined with the linear algebra and regulator computation, it leads to the following theorem:

\begin{thm}\cite[Theorem~6.1]{BF14} \label{thm:BF14}
Under ERH and smoothness heuristics, the presented algorithm computes the class group structure together with compact representations of a fundamental system of units of a number field $\KK$ of degree $n$ and discriminant $\Delta_\KK$ in time $L_\DD(a)$~with
\begin{itemize}
\item $a = \frac{2}{3}+\ve$ \quad for $\ve >0$ arbitrary small in the general case;
\item $a = \frac{1}{2}$ \qquad\, when $n \leq (\log \DD)^{3/4-\ve} $ for $\ve > 0$ arbitrary small.
\end{itemize}
\end{thm}

Figure~\ref{fig:BF14} presents the complexity of class group computations as a function of~$\alpha$, \ie the extension degree, prior to the improvements that are presented later in this article. It is based on the classification obtained in Section~\ref{sec:class} and, for $\alpha \leq \frac{1}{2}$, on the results of~\cite{GJ16}.

\begin{figure}[h]
\centering
\begin{tikzpicture}[scale=0.9]
\fill[cyan!50] (0,0) rectangle (12,3);
\fill[cyan!50] (9,0) rectangle (12,4);

\draw[thick, cyan] (6,3) -- (9,3);
\draw[thick, cyan] (9,3) -- (9,4) -- (12,4);
\draw[thick, cyan] (12,0) -- (12,4);

\fill[cyan!40] (0,3) -- (4,2) -- (6,3) -- cycle;
\fill[cyan!30] (0,3) -- (4.66,2.33) -- (6,3) -- cycle;
\fill[cyan!20] (0,3) -- (5.33,2.66) -- (6,3) -- cycle;
\draw[dashed, cyan] (0,3) -- (5.33,2.66);
\draw[dashed, cyan] (0,3) -- (4.66,2.33);
\draw[dashed, thick, cyan] (0,3) -- (6,3) -- (4,2) -- (0,3);
\draw[<->, violet] (3.5,3) node[above]{depending on $\gamma$} -- (3.5,2.5);

\node at (7.5,2) [violet]{$L_\DD\left(\frac{1}{2}\right)$};
\node at (10.5,2) [violet]{$L_\DD\left(\frac{2}{3}+\ve\right)$};
\node at (3,1.5) [violet]{$L_\DD\left(\max(\alpha,\frac{\gamma}{2})\right)$};
\node at (7.5,1) {\cite{BF14}};
\node at (10.5,1) {\cite{BF14}};
\node at (3,0.75) {\cite{GJ16}};

\draw[->] (0,0) -- (0,5); 
\node at (-0.5,5) {$a$};
\foreach \x/\xtext in {0,2/{\frac{1}{3}},3/{\frac{1}{2}},4/{\frac{2}{3}}}
{\draw (0.1cm,\x) -- (-0.1cm,\x) node[left] {$\xtext\strut$};}
\draw[->] (0,0) -- (13,0);
\foreach \x/\xtext in {0,3/{\frac{1}{4}},4/{\frac{1}{3}},6/{\frac{1}{2}},8/{\frac{2}{3}},9/{\frac{3}{4}},12/1}
{\draw (\x,0.1cm) -- (\x,-0.1cm) node[below] {$\xtext\strut$};}
\node at (13,-0.5) {$\alpha$};

\draw[loosely dashed, gray!50] (0,4) -- (9,4);
\draw[loosely dashed, cyan] (9,0) -- (9,3);
\draw[loosely dashed, cyan] (6,0) -- (6,3);
\draw[loosely dashed, cyan] (4,0) -- (4,2);
\end{tikzpicture}
\caption{Complexity obtained by prior algorithms.}
\label{fig:BF14}
\end{figure}
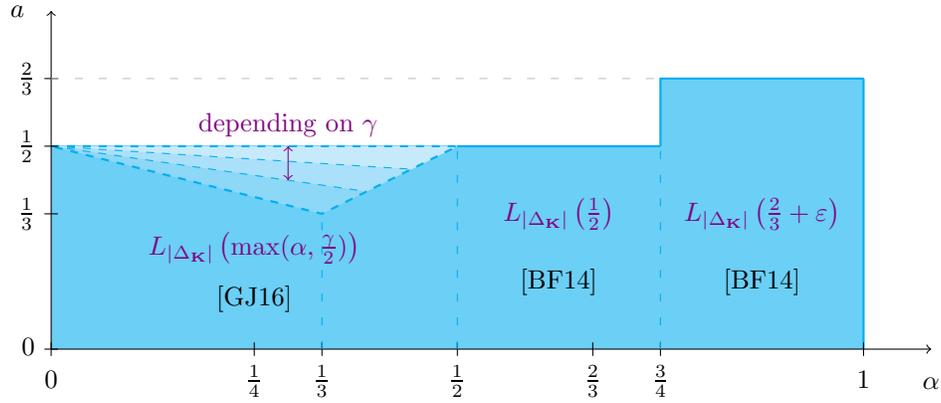

\subsection{Description of the algorithm of Biasse and Fieker} \label{sec:DescAlgo}

In~\cite{Bia14b}, and so in~\cite{BF14}, the relation collection is derived from a reduction algorithm that given an ideal $\afrak$ returns a smooth ideal $\bfrak$ that is in the same class as $\afrak$.
Then, applying this reduction to every ideal belonging to the factor base, we get the relations we are expecting.

We recall that the factor base $\Bcal = \{\pfrak_1, \dotsc, \pfrak_{|\Bcal|}\}$ consists of all prime ideals of~$\OO_\KK$ whose norm is below a bound $B = \LL_\DD(\beta,c_b)$, for $\beta \in [0,1]$ and $c_b > 0$. We also fix $\ve>0$ arbitrarily small and $\afrak$ an ideal of $\OO_\KK$.

From the ideal $\afrak$, Biasse and Fieker derive an ideal $\cfrak$ in $\OO_\KK$ by $\cfrak = \Ncal(\afrak) \cdot \afrak^{-1}$. This step consists of taking the inverse of $\afrak$, and includes a norm multiplication to keep an integral ideal. Then, similar to what Buchmann did, they choose an element $x \in \cfrak$, that is small in a certain sense, and define $\bfrak$ as the unique integral ideal that satisfies $\gen{x} = \cfrak \bfrak$. This $\bfrak$ is well-defined, as $x \in \cfrak$ implies $\gen{x} \subset \cfrak$. Finally, $\bfrak$ is in the same class as $\afrak$, as $\bfrak = \gen{x} \cfrak^{-1} = \gen{\frac{x}{\Ncal(\afrak)}} \afrak$.

\subsection*{Lattice reductions.} For finding these small elements in given ideals, it is common to consider an ideal as a lattice. For a degree-$n$ number field $\KK$, there exist $n=r_1+2r_2$ complex embeddings from $\KK$ to $\CC$.
We almost always order them in the following way: $\sigma_1,\dotsc,\sigma_{r_1}$ for the real embeddings and $\sigma_{r_1+r_2+i} = \overline{\sigma_{r_1+i}}$ for $1\leq i\leq r_2$. Hence, we get an embedding $\sigma$, called the \emph{canonical embedding}, \[\sigma : \KK \longrightarrow \RR^{r_1} \times \CC^{r_2}.\] For practical purpose, it is often considered as an $r_1+2r_2=n$-tuple of real numbers.
\begin{lem} \label{lem:DetLat}
For any integral ideal $\afrak$ of $\KK$, $\sigma(\afrak)$ is a lattice of $\RR^n$ and \[\det \sigma(\afrak) = \sqrt{\DD}\cdot\Ncal(\afrak).\]
\end{lem}

Finding small elements in a lattice is a well-studied problem. We know that this problem is exponential in the dimension if we want the smallest vector, but polynomial if we allow an exponential approximation factor. A balance has been found using BKZ algorithm: a subexponential algorithm with a subexponential approximation factor. It consists in reducing blocks of size $\beta \leq n$, so that the complexity is exponential in the block-size $\beta$. The result we use in this article is derived from the work of Micciancio and Walter~\cite{MW16}.

\begin{thm} \label{thm:BKZ} 
The smallest vector $v$ output by the BKZ algorithm with block-size~$\beta$ has a norm bounded by
\[\|v\| \quad \leq \quad \beta^{\frac{n-1}{2(\beta-1)}} \cdot (\det \Lcal)^{\frac{1}{n}}.\] The algorithm runs in time $\Poly(n,\log \|B_0\|) \left(\frac{3}{2}\right)^{\beta/2+o(\beta)}$, where $B_0$ is the input basis.
\end{thm}

\begin{proof}
The bound we get is a direct consequence of~\cite[Theorem~1]{MW16}. We only replaced the \emph{Hermite constant} $\gamma_\beta$ by an upper bound in $O\left(\beta\right)$. The cost analysis is derived from a quick study of~\cite[Algorithm~1]{MW16}, and the complexity of the \emph{Shortest Vector Problem} (SVP) is below $\left(\frac{3}{2}\right)^{\beta/2+o(\beta)}$ operations, according to~\cite{BDGL16}.
\end{proof}

The difference between the works of Buchmann and Biasse-Fieker appears in the way to choose the \emph{small} element $x$ in the ideal $\cfrak$: Biasse and Fieker replace shortest vector computations as used by Buchmann by BKZ-reductions.

\subsection*{Smoothness of ideals.} We provide in Appendix~\ref{app:smooth} a brief reminder about smoothness properties and tests, for ideals in number fields. The main assumption we need to do is the following one. It is a direct consequence of what we know for integers.

\begin{heur} \label{heur:smooth}
The probability $\Pcal(x,y)$ that an ideal of norm bounded by $x$ is \mbox{$y$-smooth} satisfies \[\Pcal(x,y) \geq e^{-u (\log u)(1+o(1))} \quad \text{for} \quad u=\frac{\log x}{\log y}.\]
\end{heur}

Because of the assumption of Heuristic~\ref{heur:smooth}, we know that we have to repeatedly select elements in $\afrak$ before finding one that leads to a smooth ideal $\bfrak$. Hence, we require a randomization process that given the ideal $\afrak$, produces as many ideals as required to guarantee to get a smooth~$\bfrak$. This is done by considering ideals of the form $\afrak \cdot \prod \pfrak_i ^{e_i}$, where the $\pfrak_i$ are prime ideals whose norms are below the smoothness bound $B$.

Then for each ideal $\tilde{\afrak} = \afrak \cdot \prod \pfrak_i ^{e_i}$, they compute the BKZ$_\beta$-reduced basis of the integral ideal $\tilde{\cfrak} = \Ncal(\tilde{\afrak}) \cdot \tilde{\afrak}^{-1}$, with a block-size~$\beta$ as determined below. This BKZ-reduction is performed on the ideal lattice $\sigma(\tilde{\cfrak})$, defined by the canonical embedding of $\tilde{\cfrak}$. As recalled earlier, it may be viewed as a lattice in $\RR^n$ using the Minkowski map.

Denoting by $x_v$ the algebraic integer corresponding to the smallest vector $v$ of the BKZ-reduced basis, they set $\tilde{\bfrak} = \gen{\frac{x_v}{\Ncal(\tilde{\afrak})}} \tilde{\afrak}$. Then $\tilde{\bfrak}$ is in the same class as $\tilde{\afrak}$ and
\begin{equation} \label{eq:BkNorm}
\Ncal(\tilde{\bfrak}) \leq \beta^{\frac{n(n-1)}{2(\beta-1)}} \sqrt{\DD}.
\end{equation}
Indeed $\Ncal(\tilde{\cfrak}) = \Ncal(\tilde{\afrak})^{n-1}$ and $\|v\| \leq \beta^{\frac{n-1}{2(\beta-1)}} \Ncal(\tilde{\cfrak})^\frac{1}{n} \DD^{\frac{1}{2n}}$ from Theorem~\ref{thm:BKZ} and Lemma~\ref{lem:DetLat}.
Then $\Ncal(\tilde{\bfrak}) \leq \left(\frac{\|v\|}{\Ncal(\tilde{\afrak})}\right)^n \Ncal(\tilde{\afrak})$ leads to the expected result.

If $\tilde{\bfrak}$ splits over the factor base $\Bcal$, then there exist integers $e_i'$ such that $\tilde{\bfrak} = \prod \pfrak_i^{e_i'}$. Thus, taking care of the randomized factor, we get that the ideal $\frac{x_v}{\Ncal(\tilde{\afrak})} \afrak$ also splits over $\Bcal$ as $\prod \pfrak_i^{e_i'-e_i}$. In the end, if $\afrak$ splits over $\Bcal$, then the principal ideal
\[\gen{\frac{x_v}{\Ncal(\afrak)}} \text{ also splits over } \Bcal.\]
Therefore we have derived a relation in the kernel of the surjective morphism defined in Equation~\eqref{eq:Morph}. If $\tilde{\bfrak}$ does not split, then we try another~$\tilde{\afrak}$. To bound the number of relations that would be sufficient, they state the following heuristic. Because at least $N=|\Bcal|$ are required, they choose the largest bound that does not increase the complexity.

\begin{heur} \label{heur:NbRel}
There exists a value $K$ that is negligible compared with $|\Bcal|$ such that collecting $K\cdot|\Bcal|$ relations suffices to obtain a relation matrix that has full-rank.
\end{heur}

Finally, using the right parameters and ideals $\afrak$ from the factor base, they get the complexities given in Theorem~\ref{thm:BF14}. With a factor base size in $L_\DD(a)$ and a block-size $(\log \DD)^a$, the overall complexity turns out to be in $L_\DD(a)$.

\subsection{Our proposition for a simpler algorithm} \label{sec:DescSimplerAlgo}

Instead of precisely studying the complexity of the algorithm of Biasse and Fieker, we rather provide a simpler version, more adapted to our problem. We focus on the relation collection in class group computations for number fields, without using information brought by the defining polynomial. Hence, ideals are viewed as lattices in $\RR^n$.

First we do not do the reduction of a specific ideal, but we take as inputs random power-products of factor-base elements. Let $k, A > 0$ be integers in $\Poly(\log \DD)$. We choose $k$ prime ideals $\pfrak_{j_1}, \dotsc, \pfrak_{j_k}$ in the factor base. For any $k$-tuple $(e_1, \dotsc, e_k) \in \{1,\dotsc, A\}^k$, we set $\afrak=\prod\limits_{i=1}^k \pfrak_{j_i}^{e_i}$ and we have
\[\Ncal\left(\afrak\right) = \Ncal\left(\prod_{i=1}^k \pfrak_{j_i}^{e_i}\right) \leq \prod_{i=1}^k \Ncal\left(\pfrak_{j_i}\right)^{e_i} \leq L_\DD\left(\beta,c_b\right)^{k \cdot A}.\]
This initialization step can be done by choosing the tuple $(e_1, \dotsc, e_k)$ uniformly at random and~$k$ prime ideals in $\Bcal$. Since from Landau's Prime Ideal Theorem~\cite{Lan03}, $|\Bcal| = L_\DD(\beta,c_b)$, the set of possible samples is large enough for our purposes. In addition, the norm of the input ideals~$\afrak$ is always polynomial in the size of the factor base.

Second we reduce the lattice defined by the ideal $\afrak$ itself, not its normalized inverse. Instead of performing the normalization explained in the previous section, we directly search for a small vector in the ideal $\afrak$ --- more precisely, in the lattice $\sigma(\afrak)$ defined by the canonical embedding. Hence we find a small vector $v$ that is the embedding of an algebraic integer $x_v$. Because $x_v$ lies in $\afrak$, there exists a unique integral ideal $\bfrak$ such that \[\gen{x_v} = \afrak \bfrak.\] 
The attentive reader should point out that the ideals $\afrak$ and $\bfrak$ do not belong to the same ideal class as before. However, this is not so important, because $\bfrak^{-1}$ for instance shares the same class with $\afrak$. Our ultimate goal is to figure out a principal ideal that is $B$-smooth, and this is achieved with our method too.

For the recovery of the algebraic integer $x_v$ associated to the vector $v$, one can make use of the transformation matrix corresponding to the variable change. Another possibility is to work directly with the conjugates and go back to the algebraic representation using round-off, as mentioned in~\cite[Section~4.2.4]{Coh93}.

Third we use the reduction algorithm described by Espitau and Joux in~\cite{EJ18}. It works on the Gram matrix of the lattice instead of the basis matrix, and requires less precision. From a practical perspective, their algorithm is able to ensure that the input precision suffices and certifies that the output is an exact reduced basis. A precision analysis similar to the one in~\cite[Section~5]{GJ16} leads to the conclusion that the required precision is polynomial in the size of the input. Indeed, we only have to replace the weight term $c^k$ by the norm of the ideal $L_\DD(\beta,c_b)^{k \cdot A}$ whose size is still polynomial in $\log \DD$.

\begin{algorithm}[h]
\caption{Deriving relations from BKZ$_\beta$-reduction}
\label{algo:DerivBKZ}
\begin{algorithmic}[1]
\REQUIRE The factor base $\Bcal$, the block-size $\beta$, the bounds $k$ and $A$ for building ideals.
\ENSURE The relations stored.
\WHILE {not enough relations are found}
\STATE Choose at random $k$ prime ideals $\pfrak_{j_1},\dotsc,\pfrak_{j_k}$ in the factor base $\Bcal$
\STATE Choose at random $k$ exponents $e_{j_1},\dotsc,e_{j_k}$ in $\{1,\dotsc,A\}$
\STATE Set $\afrak = \prod \pfrak_i^{e_i} \quad$, for $i \in \{1,\dotsc,|\Bcal|\}$, with $e_i=0$ if $i \notin \{j_1,\dotsc,j_k\}$
\STATE Find a BKZ$_\beta$-reduced basis of $\afrak$
\STATE Let $x_v$ denote the algebraic integer corresponding to the smallest vector of this basis
\STATE Set $\bfrak$ as the unique ideal such that $\gen{x_v} = \afrak \bfrak$
\IF {$\bfrak$ is $B$-smooth}
\STATE Let $e'_i$ such that $\bfrak = \prod \pfrak_i^{e'_i}$
\STATE Store the relation $\gen{x_v} = \prod \pfrak_i^{e_i+e'_i}$
\ENDIF
\ENDWHILE
\end{algorithmic}
\end{algorithm}

\begin{rem} \label{rem:rand}
Another improvement should be to test for smoothness all the elements whose norms are below the bound given by the theoretic study of BKZ reduction.
The first vector output by the BKZ reduction has norm below $\beta^{\frac{n-1}{2(\beta-1)}} \Ncal(\afrak)^{\frac{1}{n}} \DD^{\frac{1}{2n}}$ and this bound is the one we used for the complexity analysis. However, if several small vectors have their norm below this bound, then the rest of the algorithm works similarly for them, and we have saved the cost of BKZ reductions. Hence, one may try the first small linear combinations between vectors of the reduced basis output after reduction. This is only a practical improvement, because asymptotically the number of BKZ reductions performed is not taken into account (see Section~\ref{sec:Comp}).
\end{rem}

The algorithm stops when enough relations are collected. At this point, it is necessary to rely on a heuristic (as Heuristic~\ref{heur:NbRel}) in order to guarantee the result. We propose a new one that suffices for our purposes. We want the number of relations to be sufficient to generate the whole set of relations described in Equation~\eqref{eq:Morph}. We~emphasize that there exist ideals in the factor base that are more important: the ones whose norm is below the \emph{Bach bound} $12(\log \DD)^2$. Thus we consider that the matrix construction is completed when the number of relations is larger than the number of ideals that occur and when all ideals of norm below Bach's bound are involved in at least one relation. This last condition means that the submatrix built from all the relations and only those ideals must have full-rank. In comparison with Heuristic~\ref{heur:NbRel}, our relation matrix may contain all-zero columns, which correspond to ideals in the factor base that are not involved in any of the relations. By construction their norms are necessarily larger than $12(\log \DD)^2$.

\begin{heur} \label{heur:NbRel+}
There exists $K$ negligible compared with $|\Bcal|$ such that collecting $K\cdot|\Bcal|$ relations suffices to obtain a relation matrix that generates the whole lattice of relations.
\end{heur}

\subsection{Parameter settings}

We consider as input a number field $\KK \in \Dcal_{n_0,d_0,\alpha,\gamma}$, with $\alpha \geq \frac{1}{2}$. We stress that no information is needed on the size of the defining polynomial --- namely on $\gamma$ --- for this algorithm. Table~\ref{tab:Param} lists the optimal choices for the factor-base bound $B$ and the block-size $\beta$ depending on $\alpha$, with a transition at $\alpha = \frac{3}{4}$ (as already mentioned by Biasse and Fieker). The parameter $c_b > 0$ is going to be determined later, based on the complexity analysis.

\renewcommand{\arraystretch}{1}
\begin{table}[h]
\[\begin{array}{c|cc}
& B & \beta\\
\hline
\multirow{2}{*}{$\frac{1}{2} \leq \alpha \leq \frac{3}{4}$} & \multirow{2}{*}{$\LL_\DD\left(\frac{1}{2}, c_b\right)$} & \multirow{2}{*}{$\left(\log \DD\right)^{\frac{1}{2}}$} \\ & & \\
\multirow{2}{*}{$\frac{3}{4} < \alpha \leq 1$} & \multirow{2}{*}{$\LL_\DD\left(\frac{2\alpha}{3}, c_b\right)$} & \multirow{2}{*}{$\left(\log \DD\right)^{\frac{2\alpha}{3}}$} \\ & & \\
\end{array}\]
\caption{Optimal choices for the factor-base bound and block-size depending on the extension degree.}
\label{tab:Param}
\end{table}

\section{Complexity analyses} \label{sec:Comp}

\subsection{The case $\alpha \leq \frac{3}{4}$} \label{sec:med}

According to~\cite{BF14}, when $\alpha \leq \frac{3}{4}$, we know that our algorithm should run in time $L_\DD \left( \frac{1}{2}, c_1 \right)$. We provide a detailed analysis to find an explicit expression for the constant $c_1$. Let $\KK$ be a number field belonging to~$\Dcal_{n_0,d_0,\alpha,\gamma}$ with $\alpha \in \left[\frac{1}{2},\frac{3}{4}\right] $, $\gamma \geq 1- \alpha$, $d_0 >0$ and $n_0>1$. The factor base $\Bcal$ is fixed as the set of all prime ideals of norm below $B = \LL_\DD\left(\frac{1}{2}, c_b\right)$, with $c_b > 0$ to be determined, and the block-size used in BKZ-reduction is $\beta = (\log \DD)^{\frac{1}{2}}$, according to Table~\ref{tab:Param}.

First, we analyze the BKZ-reduction. Before looking at the output, we focus on the cost of the reduction. By construction of ideal $\afrak$ --- see Section~\ref{sec:DescSimplerAlgo} --- its norm is polynomial in $L_\DD \left(\frac{1}{2}\right)$. Theorem~\ref{thm:BKZ} states that BKZ-reduction runs in time $\Poly\left(n,\log \Ncal(\afrak)\right) \cdot 2^{O(\beta)}$. Because the norm of $\afrak$ is upper bounded, it only remains to bound the factor $2^{O(\beta)}$. Denoting by $C$ the constant in the~$O$, we asymptotically obtain
$\log 2^{O(\beta)} = C\cdot\log 2 \cdot (\log \DD)^{\frac{1}{2}} \leq c (\log \DD)^{\frac{1}{2}} (\log \log \DD)^{\frac{1}{2}}$ for any constant $c>0$. Thus, we have shown that the runtime of the reduction algorithm is below~$L_\DD \left(\frac{1}{2},c\right)$ for every $c>0$.

Second, we estimate the norm of the new ideal $\bfrak$ built from the smallest vector returned by the reduction algorithm. From Theorem~\ref{thm:BKZ} and Lemma~\ref{lem:DetLat}, we deduce that the smallest vector $v$ of the BKZ$_\beta$-reduced basis has a norm that satisfies $\|v\| \leq \beta^{\frac{n-1}{2(\beta-1)}} \Ncal(\afrak)^{\frac{1}{n}} \DD^{\frac{1}{2n}}$. As $\Ncal(x_v) \leq \|v\|^n$, we directly derive that the norm of $\bfrak$ is upper bounded by $\Ncal(\bfrak) \leq \beta^{\frac{n(n-1)}{2(\beta-1)}} \sqrt{\DD}$ so that we deduce\footnote{Note that it is the same bound as in Equation~\eqref{eq:BkNorm}. Our adjustments in the algorithm do not affect this bound.}
\begin{eqnarray*}
\log\Ncal(\bfrak) &\leq& \frac{1}{2}\log \DD + \frac{n_0^2}{4} (\log \DD)^{2\alpha-\frac{1}{2}} (\log \log \DD)^{1-2\alpha} \\
&\leq& \frac{1}{2}\log \DD + c (\log \DD)^{2\alpha-\frac{1}{2}} (\log \log \DD)^{1-2\alpha+\frac{1}{2}} \quad \mbox{for all $c>0$} \\
&\leq& \frac{1}{2}\log \DD \big(1+o(1)\big) \\
\Longrightarrow \quad \Ncal(\bfrak) & \leq & L_\DD \left(1,\frac{1}{2}\right).
\end{eqnarray*}

Third, we have to express the probability for such a $\bfrak$ to be $B$-smooth. Assuming Heuristic~\ref{heur:smooth} allows us to get a probability of \[L_\DD \left(\frac{1}{2}, \frac{1}{4c_b}\right)^{-1}.\]
Hence, on average, testing $L_\DD \left(\frac{1}{2}, \frac{1}{4c_b}\right)$ ideals $\afrak$ leads to a single ideal $\bfrak$ that is $B$-smooth and thus to one relation. Assuming Heuristic~\ref{heur:NbRel+}, we need to find $L_\DD\left(\frac{1}{2}, c_b\right)$ relations. This requires testing for smoothness \[L_\DD\left(\frac{1}{2}, \frac{1}{4c_b} + c_b\right)\] ideals. From Appendix~\ref{app:smooth}, we know that each test costs $L_{L_\DD\left(\frac{1}{2}\right)}\left(\frac{1}{2}\right) = L_\DD\left(\frac{1}{4}\right)$, which is negligible. The reduction step, whose runtime is below $L_\DD \left(\frac{1}{2},c\right)$ for every $c > 0$, is also negligible. Hence the global complexity of the relation collection step is given by the number of ideals that we test, that is \[L_\DD\left(\frac{1}{2}, \frac{1}{4c_b} + c_b\right).\]

\subsection*{Complexity for the class group computation.}
Now that we know the complexity of the collection step, we look at the remaining parts of the computation to get the class group structure, in order to determine the best $c_b$. The relations are stored in a matrix of size $K\cdot N\times N$, with $N=|\Bcal|=L_\DD \left(\frac{1}{2},c_b\right)$. 
The results regarding linear algebra, precision and regulator computation are already studied by Biasse and Fieker in~\cite{BF14}. They show~\cite[Proposition~4.1]{BF14} that the class group structure is inferred from the relation matrix in time $L_\DD \left(\frac{1}{2},(\omega + 1)c_b\right)$, where~$\omega$ denotes the matrix multiplication exponent. This result essentially relies on the HNF algorithm of Storjohann and Labahn~\cite[Theorem~12]{SL96}.

The best choice for $c_b$ --- \ie the one that minimizes the complexity --- follows from balancing the runtimes of the collection and linear algebra phases. Thus the parameter $c_b > 0$ should satisfy
\[\frac{1}{4c_b}+c_b = (\omega +1)c_b \quad \Longleftrightarrow \quad c_b = \frac{1}{2\sqrt{\omega}}.\]

\begin{thm} \label{thm:med}
Assuming ERH and Heuristics~\ref{heur:smooth} and~\ref{heur:NbRel+}, for every number field~$\KK$ that belongs to $\Dcal_{n_0,d_0,\alpha,\gamma}$ with $\alpha \in \left[\frac{1}{2},\frac{3}{4}\right]$, our algorithm computes the class group structure and the regulator with runtime \[L_\DD \left(\frac{1}{2},\frac{\omega+1}{2\sqrt{\omega}}\right).\]
\end{thm}

\begin{rem}
We recall that $\omega$ denotes the exponent arising in the complexity of matrix multiplication.
The smallest known value is $\omega = 2.3728639$ (see~\cite{Gal14}) which correspond to the value $1.095$ for the second constant. In practice, we use the Strassen algorithm~\cite{Str69} where \mbox{$\omega = \log_2 7 \approx 2.807$}, leading to the second-constant value $1.136$.
\end{rem}

\subsection{The case $\alpha > \frac{3}{4}$} \label{sec:big}

We follow the same path as in the previous case although some adjustments are made. We start by mentioning that our final complexity is much better than the one announced in~\cite{BF14}: we manage to replace the first constant $\frac{2}{3} + \ve$ by $\frac{2\alpha}{3}$, which is always smaller, particularly when~$\alpha$ is close to $\frac{3}{4}$. Furthermore, our second constant can be chosen arbitrarily small, which we denote by $L_\DD \left(\frac{2\alpha}{3}, o(1)\right)$.

This time, $\KK$ belongs to $\Dcal_{n_0,d_0,\alpha,\gamma}$ with $\alpha \in \left(\frac{3}{4},1\right]$. The smoothness bound is fixed to \mbox{$B=\LL_\DD\left(\frac{2\alpha}{3},c_b\right)$}, $c_b > 0$, and the block-size is $\beta = (\log \DD)^{\frac{2\alpha}{3}}$. The bound on the norms $\Ncal(\afrak)$ is polynomial in $L_\DD \left(\frac{2\alpha}{3}\right)$, because of the parameters we used for constructing the ideals $\afrak$. In the same way as in Section~\ref{sec:med}, we show that the runtime of the reduction algorithm is below $L_\DD \left(\frac{2\alpha}{3},c\right)$ for every $c > 0$. The bound we derive for the norm of the new ideal built is
\begin{eqnarray*}
\log\Ncal(\bfrak) &\leq& \frac{1}{2}\log \DD + \frac{\alpha n_0^2}{3} (\log \DD)^{\frac{4\alpha}{3}} (\log \log \DD)^{1-2\alpha} \\
&\leq& \frac{1}{2}\log \DD + c (\log \DD)^{\frac{4\alpha}{3}} (\log \log \DD)^{1-\frac{4\alpha}{3}} \qquad \mbox{for all $c>0$} \\
&\leq& c (\log \DD)^{\frac{4\alpha}{3}} (\log \log \DD)^{1-\frac{4\alpha}{3}} \qquad \mbox{for all $c>0$}.
\end{eqnarray*}

Assuming Heuristic~\ref{heur:smooth} and fixing any $c>0$, if we take $c_b = \sqrt{\frac{2\alpha c}{3}}$ in the definition of $B$, then the probability for ideal $\bfrak$ to be $B$-smooth is \[L_\DD \left(\frac{2\alpha}{3}, c_b\right)^{-1}.\] Hence we conclude that testing $L_\DD\left(\frac{2\alpha}{3}, 2c_b \right)$ ideals suffices for the entire collection phase.
Again, the runtime $L_\DD\left(\frac{\alpha}{3}\right)$ to perform a single smoothness test can be neglected.

\subsection*{Complexity for the class group computation.}
The global complexity of the class group computation follows directly, because the runtime of the linear algebra step is obtained by multiplying the second constant $2c_b$ by a constant factor $\omega+1$.
As the constant $c_b$ could be chosen arbitrarily small (but positive), we get the following theorem.

\begin{thm} \label{thm:big}
Assuming ERH and Heuristics~\ref{heur:smooth} and~\ref{heur:NbRel+}, for every number field~$\KK$ that belongs to $\Dcal_{n_0,d_0,\alpha,\gamma}$ with $\alpha \in \left(\frac{3}{4},1\right]$, our algorithm computes the class group structure and the regulator with runtime \[L_\DD \left( \frac{2\alpha}{3},o(1) \right).\]
\end{thm}

We can now update Figure~\ref{fig:BF14}, by taking into account the results of Theorems~\ref{thm:med} and~\ref{thm:big}. This is presented in Figure~\ref{fig:LD}.

\begin{figure}[h]
\centering
\begin{tikzpicture}[scale=0.9]
\fill[cyan!50] (0,0) rectangle (12,3);
\fill[cyan!50] (9,0) -- (9,3) -- (12,4) -- (12,0) -- cycle;

\draw[thick, cyan] (6,3) -- (9,3);
\draw[thick, cyan] (9,3) -- (12,4);
\draw[thick, cyan] (12,0) -- (12,4);

\fill[cyan!40] (0,3) -- (4,2) -- (6,3) -- cycle;
\fill[cyan!30] (0,3) -- (4.66,2.33) -- (6,3) -- cycle;
\fill[cyan!20] (0,3) -- (5.33,2.66) -- (6,3) -- cycle;
\draw[dashed, cyan] (0,3) -- (5.33,2.66);
\draw[dashed, cyan] (0,3) -- (4.66,2.33);
\draw[dashed, thick, cyan] (0,3) -- (6,3) -- (4,2) -- (0,3);
\draw[<->, violet] (3.5,3) node[above]{depending on $\gamma$} -- (3.5,2.5);

\node at (3,1.5) [violet]{$L_\DD\left(\max(\alpha,\frac{\gamma}{2})\right)$};
\node at (7.5,1.5) [violet]{$L_\DD\left(\frac{1}{2},\frac{\omega-1}{2\sqrt{\omega}}\right)$};
\node at (10.5,1.5) [violet]{$L_\DD\left(\frac{2\alpha}{3},o(1)\right)$};
\node at (3,0.75) {\cite{GJ16}};

\draw[->] (0,0) -- (0,5); 
\node at (-0.5,5) {$a$};
\foreach \x/\xtext in {0,2/{\frac{1}{3}},3/{\frac{1}{2}},4/{\frac{2}{3}}}
{\draw (0.1cm,\x) -- (-0.1cm,\x) node[left] {$\xtext\strut$};}
\draw[->] (0,0) -- (13,0);
\foreach \x/\xtext in {0,3/{\frac{1}{4}},4/{\frac{1}{3}},6/{\frac{1}{2}},8/{\frac{2}{3}},9/{\frac{3}{4}},12/1}
{\draw (\x,0.1cm) -- (\x,-0.1cm) node[below] {$\xtext\strut$};}
\node at (13,-0.5) {$\alpha$};

\draw[loosely dashed, gray!50] (0,4) -- (12,4);
\draw[loosely dashed, cyan] (9,0) -- (9,3);
\draw[loosely dashed, cyan] (6,0) -- (6,3);
\draw[loosely dashed, cyan] (4,0) -- (4,2);
\end{tikzpicture}
\caption{Complexity obtained by our algorithms.}
\label{fig:LD}
\end{figure}
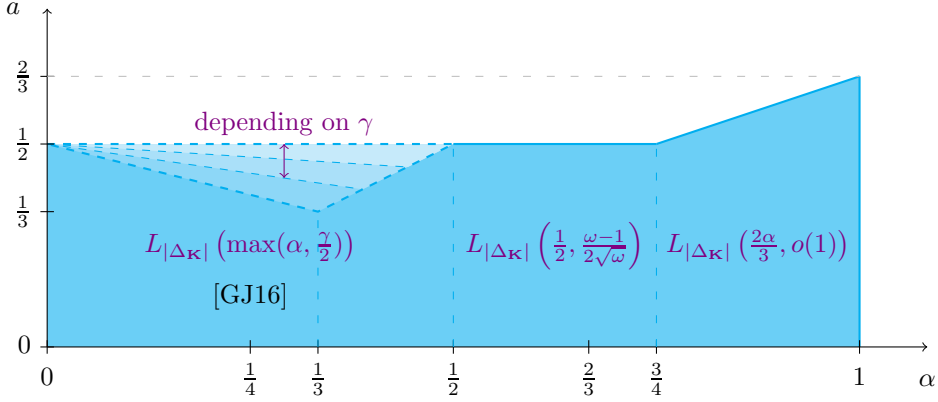

\section{Using HNF to get an even smaller complexity} \label{sec:cheon}

We have a complexity between $L_\DD\left(\frac{1}{2}\right)$ and $L_\DD\left(\frac{2}{3}\right)$, which grows linearly for classes $\Dcal$ with $\alpha \geq \frac{3}{4}$. We want to reduce this worst case using Cheon's trick, that allows to output a shorter vector than in the general case. It relies on the reduction of a sublattice that has smaller dimension than the full lattice, provided that the input lattice has small discriminant. This method seems to be folklore, but his note~\cite{CL15} gives a detailed analysis and we refer to it as \textit{Cheon's trick}.

\begin{lem} \label{lem:Cheon}
Given $(b_1,\dotsc,b_n)$ a basis in HNF of an $n$-dimensional lattice $\Lcal \subset~\RR^n$, we have, for any $1 \leq i < n$, \[\det \left[b_1,\dotsc,b_i\right] \leq \det \left[b_1,\dotsc,b_{i+1}\right].\]
In particular, for any sublattice $\Lcal'$ generated by the $m$ first vectors $b_1,\dotsc,b_m$, we have \[\det \Lcal' \leq \det \Lcal.\]
\end{lem}

Remark that both the $n$-th root of the determinant and an exponential factor in $n$ appear in the bound of Theorem~\ref{thm:BKZ}. In most cases, the term with the determinant prevails. However, when the determinant is small, the approximation factor can be larger. The idea behind Cheon's trick is then to reduce a lattice of smaller dimension in order to reduce this approximation factor. We fix the block-size $\beta \leq n$ and look at the output of BKZ performed on the sublattice $\Lcal'$ generated by the $m$ first vectors $b_1,\dotsc,b_m$ of an HNF basis. From Lemma~\ref{lem:Cheon}, we get
\[\|v\| \leq \beta^{\frac{m}{2\beta}} \cdot (\det \Lcal')^{\frac{1}{m}} \leq \beta^{\frac{m}{2\beta}} \cdot (\det \Lcal)^{\frac{1}{m}}.\] \bigskip

The condition we require on the determinant of the lattice is $\det \Lcal \leq \beta^{\frac{n^2}{2\beta}}$: otherwise, for every $m \leq n$, the term $(\det \Lcal)^{\frac{1}{m}}$ is dominating. Assuming that $\det \Lcal \leq \beta^{\frac{n^2}{2\beta}}$, we identify the optimal sub-dimension $m$ in $\{\beta,\dotsc,n\}$ depending on~$\beta$ that minimizes this upper bound: it corresponds to the balance between the two factors, that is $m=\left\lfloor \sqrt{2\beta \log_\beta (\det \Lcal)}\right\rceil$. We fix $m$ to this value and we obtain the following corollary.

\begin{cor} \label{cor:Cheon}
For any integer lattice $\Lcal \subset \RR^n$ of rank $n$ such that $\det \Lcal \leq \beta^{\frac{n^2}{2\beta}}$, using BKZ reduction with block-size $\beta$ along with Cheon's trick permits to output a short vector $v$ that satisfies \[\log_\beta \|v\| \leq \sqrt{\frac{2}{\beta}\log_\beta(\det \Lcal)}\big(1+o(1)\big).\] This algorithm runs in time $\Poly(n,\log \|B_0\|) \cdot \left(\frac{3}{2}\right)^{\beta/2+o(\beta)}$, with $B_0$ the input basis.
\end{cor}

\begin{proof}
We consider the sublattice of dimension $m$, for $m$ as defined above. The condition on the determinant of $\Lcal$ ensures that our value of $m$ is effectively lower than $n$. Then, by Theorem~\ref{thm:BKZ} and Lemma~\ref{lem:Cheon}, we have
\[\|v\| \leq \beta^{\frac{m}{2\beta}} \cdot (\det \Lcal)^{\frac{1}{m}} = \beta^{\sqrt{(2/\beta)\log_\beta(\det \Lcal)}\big(1+o(1)\big)},\]
which yields the announced result --- the $\left(1+o(1)\right)$ factor appears because of the integer approximation of $m$.
\end{proof}

As shown in this section, the complexity we are able to reach with this method is $L_\DD\left(\frac{2\alpha+1}{5}\right)$. It varies linearly between $L_\DD\left(\frac{1}{2}\right)$ and $L_\DD\left(\frac{3}{5}\right) < L_\DD\left(\frac{2}{3}\right)$. Hence, we fix the smoothness bound $B=\LL_\DD\left(\frac{2\alpha+1}{5},c_b\right)$, with $c_b >0$ to be determined. Also, the block-size used for BKZ-reductions is set to $\beta=\left(\log \DD\right)^{\frac{2\alpha+1}{5}}$. Overall, the path followed by this improved version of our algorithm is essentially similar to the one described in Section~\ref{sec:DescSimplerAlgo}. We only mention the adjustments in the reminder of the section.

First, we need to work with an integral lattice. Indeed as we begin by computing the HNF of the lattice, it must be defined over $\ZZ$. This is not a problem, as we already mentioned. We know that the required precision is polynomial in the size of the entries. Practically, we approximate the Gram matrix and use the implementation of~\cite{EJ18}. We also mention the special case of totally real number fields where no approximation are required as the Gram matrix is integral.

Second, to ensure that the hypothesis of Corollary~\ref{cor:Cheon} is satisfied, we need a bound on the determinant of the input lattice. As we want a lattice with \emph{small determinant} as input, we first perform a rough reduction, using the classical BKZ algorithm --- that is without Cheon's trick. Given an ideal $\afrak$ constructed as above as a power-product of elements in the factor base and denoting by $v$ the first vector of the BKZ-reduced basis, we define the ideal $\bfrak$ as the unique integral ideal that satisfies \[\gen{x_v} = \afrak \bfrak.\]

Thanks to the analysis presented in Section~\ref{sec:big}, we know that the norm of this ideal $\bfrak$ is upper bounded by $L_\DD\left(\frac{8\alpha-1}{5}\right)$. We are in the case $\alpha > \frac{3}{4}$, so that $\frac{8\alpha-1}{5} > 1$. The determinant of the lattice corresponding to the canonical embedding of $\bfrak$ is $\Ncal(\bfrak) \cdot \sqrt{\DD}$. Hence we cannot expect that this quantity is smaller than $L_\DD(1)$, so we look for an ideal $\bfrak$ that is $\widetilde{B}$-smooth for $\widetilde{B} = \LL_\DD(1,1)$. According to Proposition~\ref{prop:CostECM}, each smoothness test costs \[L_{L_\DD(1)}\left(\frac{1}{2}\right) = L_\DD\left(\frac{1}{2}\right)\] and assuming Heuristic~\ref{heur:smooth}, testing about $L_\DD\left(\frac{8\alpha-6}{5}\right)$ ideals suffices on average. In addition, the number of ideals in every smooth decomposition is upper bounded by $\left(\log \DD\right)^{\frac{8\alpha-6}{5}}\big(1+o(1)\big)$. The complete runtime of this smoothness phase is in $L_\DD\left(\frac{1}{2}\right)$, as $0 < \frac{8\alpha-6}{5} < \frac{2}{5}$, which is outweighed by the initial BKZ$_\beta$ reduction, whose cost is $L_\DD\left(\frac{2\alpha+1}{5},o(1)\right)$.

In the end, we have ideals $\bfrak_1,\dotsc,\bfrak_l$ whose canonical embeddings have determinant in $\LL_\DD(1,1) < \beta^{\frac{n^2}{2\beta}} = L_\DD\left(\frac{8\alpha-1}{5}\right)$ and which satisfy $\prod \bfrak_i = \bfrak$. We notice that for the application of Corollary~\ref{cor:Cheon}, the lower bound $\LL_\DD(1,1)$ does not have to be reached. However, as the quality of the output relies on this quantity --- a factor $\log_\beta \det \Lcal$ appears in the exponent --- we minimize it in order to get the best possible output.

For each ideal lattice $\sigma(\bfrak_i)$, we may apply Cheon's trick combined with BKZ-reduction. As for all $i$ it is the case that $\log \left(\det \sigma(\bfrak_i)\right) = C \log \DD$ for a $C>0$, a small vector $v_i$ is found with norm satisfying
\begin{align*}
\log \|v_i\| \;
&\leq \; \left(\frac{2C \log \DD}{(\log \DD)^{\frac{2\alpha+1}{5}}\log\left((\log \DD)^{\frac{2\alpha+1}{5}}\right)}\right)^{\frac{1}{2}}\!\!\!\log\!\left(\!(\log \DD)^{\frac{2\alpha+1}{5}}\!\right)\big(1+o(1)\big) \\
&\leq \;\; \sqrt{\frac{2C(2\alpha+1)}{5}}  \left(\log \DD\right)^{\frac{2-\alpha}{5}}\left(\log \log \DD\right)^{\frac{1}{2}} \big(1+o(1)\big) \\
&\leq \;\; c \left(\log \DD\right)^{\frac{2-\alpha}{5}}\left(\log \log \DD\right)^{1-\frac{2-\alpha}{5}} \qquad \text{for every constant $c>0$.} 
\end{align*}

As we did for the earlier analyses, we bound the norm of the algebraic integer $x_{v_i}$ associated to the vector $v_i$. We obtain the inequality $\Ncal\left(\gen{x_{v_i}}\right) \leq L_\DD\left(\frac{4\alpha+2}{5},c\right)$ for every $c>0$. In addition, there exist integral ideals~$\cfrak_i$ such that $\gen{x_{v_i}} = \bfrak_i \cfrak_i$ for all $i$. As the norm of $\bfrak_i$ is less than $\LL_\DD(1,1)$, we deduce that for each $i$, the norm of the ideal~$\cfrak_i$ satisfies \[\Ncal\left(\cfrak_i\right) \quad \leq \quad L_\DD\left(\frac{4\alpha+2}{5}, o(1)\right).\]

Denoting by $c$ the arbitrarily small non-negative constant that arises in the $o(1)$, we follow the same argument as in Section~\ref{sec:big}. By fixing $c_b = \sqrt{\frac{(2\alpha+1)c}{5}}$, we deduce that the probability for each $\cfrak_i$ to be $B$-smooth is \[L_\DD\left(\frac{2\alpha+1}{5}, c_b\right)^{-1}.\]

Hence we conclude that testing $L_\DD\left(\frac{2\alpha+1}{5}, 2c_b\right)$ ideals suffices to complete the relation collection. Indeed, we have to test $L_\DD\left(\frac{2\alpha+1}{5}, c_b\right)$ ideals for each ideal~$\bfrak_i$ and given an ideal $\bfrak$ as input, the number of factors $\bfrak_i$ is polynomial. Finally, assuming Heuristic~\ref{heur:NbRel+}, we require $L_\DD\left(\frac{2\alpha+1}{5}, c_b\right)$ relations, which leads to the runtime stated above for the relation collection.

\subsection*{Complexity for the class group computation.}
Again, as in Section~\ref{sec:big}, the final complexity for the class group computation follows directly and we get the following theorem.

\bigskip

\begin{thm} \label{thm:big+cheon}
Assuming ERH and Heuristics~\ref{heur:smooth} and~\ref{heur:NbRel+}, for every number field~$\KK$ that belongs to $\Dcal_{n_0,d_0,\alpha,\gamma}$ with $\alpha \in \left(\frac{3}{4},1\right]$, our algorithm computes the class group structure and the regulator with runtime \[L_\DD\left(\frac{2\alpha+1}{5},o(1)\right).\]
\end{thm}

This new result allows to reduce the slope of the increasing line appearing in our complexity figures. The worst complexity now becomes $L_\DD\left(\frac{3}{5},o(1)\right)$. This result is displayed in Figure~\ref{fig:LD+cheon}.

\begin{figure}[h]
\centering
\begin{tikzpicture}[scale=0.9]
\fill[cyan!50] (0,0) rectangle (12,3);
\fill[cyan!50] (9,0) -- (9,3) -- (12,3.6) -- (12,0) -- cycle;

\draw[thick, cyan] (6,3) -- (9,3);
\draw[thick, cyan] (9,3) -- (12,3.6);
\draw[thick, cyan] (12,0) -- (12,3.6);

\fill[cyan!40] (0,3) -- (4,2) -- (6,3) -- cycle;
\fill[cyan!30] (0,3) -- (4.66,2.33) -- (6,3) -- cycle;
\fill[cyan!20] (0,3) -- (5.33,2.66) -- (6,3) -- cycle;
\draw[dashed, cyan] (0,3) -- (5.33,2.66);
\draw[dashed, cyan] (0,3) -- (4.66,2.33);
\draw[dashed, thick, cyan] (0,3) -- (6,3) -- (4,2) -- (0,3);
\draw[<->, violet] (3.5,3) node[above]{depending on $\gamma$} -- (3.5,2.5);

\node at (3,1.5) [violet]{$L_\DD\left(\max(\alpha,\frac{\gamma}{2})\right)$};
\node at (7.5,1.5) [violet]{$L_\DD\left(\frac{1}{2},\frac{\omega-1}{2\sqrt{\omega}}\right)$};
\node at (10.5,1.5) [violet]{$L_\DD\left(\frac{2\alpha+1}{5},o(1)\right)$};
\node at (3,0.75) {\cite{GJ16}};

\draw[->] (0,0) -- (0,5); 
\node at (-0.5,5) {$a$};
\foreach \x/\xtext in {0,2/{\frac{1}{3}},3/{\frac{1}{2}},3.6/{\frac{3}{5}}}
{\draw (0.1cm,\x) -- (-0.1cm,\x) node[left] {$\xtext\strut$};}
\draw[->] (0,0) -- (13,0);
\foreach \x/\xtext in {0,3/{\frac{1}{4}},4/{\frac{1}{3}},6/{\frac{1}{2}},8/{\frac{2}{3}},9/{\frac{3}{4}},12/1}
{\draw (\x,0.1cm) -- (\x,-0.1cm) node[below] {$\xtext\strut$};}
\node at (13,-0.5) {$\alpha$};

\draw[loosely dashed, gray!50] (0,3.6) -- (12,3.6);
\draw[loosely dashed, cyan] (9,0) -- (9,3);
\draw[loosely dashed, cyan] (6,0) -- (6,3);
\draw[loosely dashed, cyan] (4,0) -- (4,2);
\end{tikzpicture}
\caption{Complexity obtained by our algorithms and Cheon's trick.}
\label{fig:LD+cheon}
\end{figure}
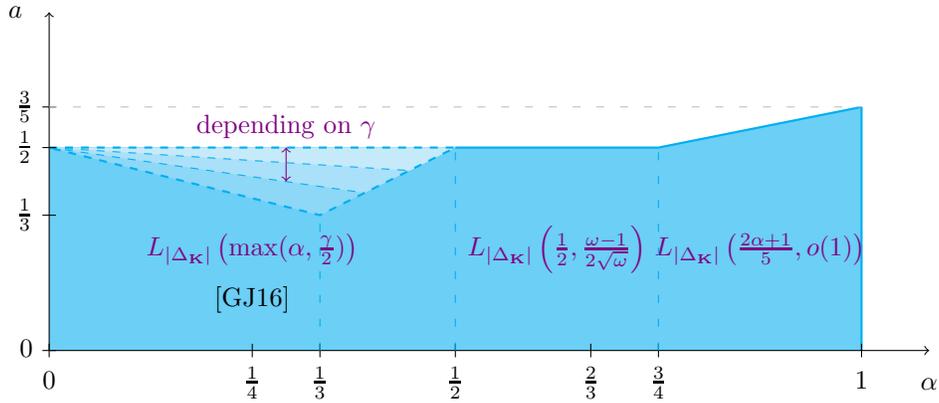

\bibliographystyle{amsalpha}
\bibliography{Biblio}

\appendix

\section{Smoothness properties} \label{app:smooth}

\subsection{Smooth integers}

The smoothness of an integer is another way to evaluate its size which depends on its prime factors.

\begin{defin}
For an integer $B \in \NN$, we say that an integer is $B$-smooth if all its prime factors are below $B$. The bound $B$ is then often called a \emph{smoothness bound}.
\end{defin}

\subsection*{Smoothness probability.}
Let us denote by $\Pcal(x,y)$ the probability that an integer~$x$ is $y$-smooth, that means all prime factors of $x$ are less than or equal to $y$. Dickman was the first one to address the question of asymptotic formulae in~\cite{Dic30}. Before stating his result, we introduce the Dickman \emph{rho}-function, defined over $\RR^+$ as the unique continuous function that satisfies $u\rho'(u) + \rho(u-1) = 0$ with initial condition $\rho(u)=1$ for $u \in [0,1]$.

\begin{prop}
For any fixed $u > 0$, we have \[\lim_{x \to \infty} \Pcal \left(x,x^{1/u}\right) = \rho (u).\]
\end{prop}

\begin{proof}
This result appears in the work of Dickman~\cite{Dic30} and in the survey written later by Hildebrand and Tenenbaum~\cite{HT93}. The latter also showed~\cite[Corollary~1.3]{HT93} that when $u$ is large enough, $\rho(u)$ may be approximated by $u^{-u(1+o(1))}$.
\end{proof}

The main drawback of that previous result is that $u$ has to be fixed: it cannot depend on $x$. This issue is covered by the stronger result of Canfield, Erd{\H{o}}s, and Pomerance in~\cite{CEP83}:

\begin{thm} \label{thm:CEP}
For every $\ve>0$, there exists a constant $C_\ve$ such that for all $x \geq 1$ and $3 \leq u \leq (1-\ve)\frac{\log x}{\log \log x}$, we have
\[\Pcal(x,x^{1/u}) \geq e^{-u\left(\log u + \log \log u -1 + \frac{\log \log u -1}{\log u} + E(x,u)\right)},\]
where \[\left|E(x,u)\right| \leq C_\ve \left(\frac{\log \log u}{\log u}\right)^2.\]
\end{thm}

Eventually, we can express $\Pcal(x,y)$ by fixing $u$ such that $u = \frac{\log x}{\log y}$ and substitute in the last expression. We obtain \[\Pcal(x,y) = u^{-u(1+o(1))},\] which we already have from Dickman's work.

\begin{cor}
Assuming that $x= \LL_N(\alpha_1,c_1)$, $y= \LL_N(\alpha_2,c_2)$, and $\alpha_1 > \alpha_2$, Theorem~\ref{thm:CEP} can be expressed as
\[\Pcal(x,y) \geq L_N\left(\alpha_1-\alpha_2,(\alpha_1-\alpha_2)\frac{c_1}{c_2}\right)^{-1}.\]
\end{cor}

\subsection*{Smoothness tests.}
Now we have estimated the ratio of smooth numbers below~$N$ to $N$, it remains to give a way to recognize them. We need to introduce \emph{smoothness tests}. The first idea one may have is considering the complete factorization. Once we know the prime decomposition of an integer, it is easy to recognize if the number is smooth with respect to some smoothness bound. The best algorithm for factoring an integer $N$ is currently the \emph{Number Field Sieve} (NFS) and has runtime in $L_N\left(\frac{1}{3}, \sqrt[3]{\frac{64}{9}}\right)$ --- see~\cite{LLMP90} for more details.

However it seems reasonable that, given a smoothness bound $B$, to test if an integer is $B$-smooth or not has a complexity that essentially depends on $B$, and not so much on the input integer. Such an algorithm exists and is derived from the \emph{Elliptic Curve Method}, introduced by Lenstra in~\cite{Len87} for factoring integers. It provides a Monte-Carlo algorithm whose heuristic complexity is given in the following proposition.

\begin{prop} \label{prop:CostECM}
For a given smoothness bound $B$ and an integer $N$, ECM finds the $B$-smooth part of $N$ in time \[\left(\log N\right)^2 \cdot L_B\left(\frac{1}{2},\sqrt{2}\right),\]
where the factor $\left(\log N\right)^2$ comes from the multiplication of two $N$-bits integers.
\end{prop}

\subsection{Smooth ideals} \label{sec:SmoothId}

For our purposes, we need to extend these results on smoothness to ideals. 

\begin{defin}
For an integer $B \in \NN$, we say that an ideal $\afrak$ is $B$-smooth if all its prime factors have a norm below $B$.
\end{defin}

Scourfield substantially shows in~\cite{Sco04} that the results of Dickman can be generalized to number fields. However, as in the case of integers, this does not suffice and we need a stronger assumption, which we formulate as Heuristic~\ref{heur:smooth}:

\begin{heur*}
The probability $\Pcal(x,y)$ that an ideal of norm bounded by $x$ is $y$-smooth satisfies \[\Pcal(x,y) \geq e^{-u (\log u)(1+o(1))} \quad \text{for} \quad u=\frac{\log x}{\log y}.\]
\end{heur*}

We stress that this is the exact correspondence of what have been proven for integers. This heuristic already appears in the work of Biasse and Fieker~\cite[Heuristic~1]{BF14} about class group computation. The previous heuristic admits a neat rewriting in terms of the handy $L$-notation:

\begin{cor} \label{cor:smooth}
Assuming that $x= \LL_\DD(\alpha_1,c_1)$, $y= \LL_\DD(\alpha_2,c_2)$, and $\alpha_1 > \alpha_2$, Heuristic~\ref{heur:smooth} can be expressed as
\[\Pcal(x,y) \geq L_\DD\left(\alpha_1-\alpha_2,(\alpha_1-\alpha_2)\frac{c_1}{c_2}\right)^{-1}.\]
\end{cor}

Note that Seysen~\cite{Sey87} proved in 1985 a similar result for quadratic number fields. For arbitrary degree, it remains conjectural, even under ERH.

\subsection*{Smoothness tests.}
Testing smoothness for ideals is not very complicated, assuming that we know how to test smoothness for integers. Indeed, given $B \in \NN$, if $\afrak$ is $B$-smooth, then in particular its norm $\Ncal(\afrak)$ is $B$-smooth. Therefore, testing smoothness for ideals essentially amounts to testing smoothness for ideal norms. Computing the norm of an ideal is easy and has a polynomial runtime in both the extension degree and the size of the norm. Once we know the prime numbers appearing in the norm, it suffices to find the valuations at the prime ideals above them. A way to figure out these valuations is explained in~\cite[Section~4.8.3]{Coh93}. The algorithm described also has a complexity that is polynomial in the extension degree and the size of the prime number $p$.

Finally, the runtime of ideal smoothness tests is the same as integer smoothness tests: \[\Poly\left(n,\log \Ncal(\afrak)\right) \cdot L_B\left(\frac{1}{2},\sqrt{2}\right),\]
where $n$ is the extension degree of the field and $\Ncal(\afrak)$ the norm of the ideal we want to test.
\end{document}